\documentclass[12pt, reqno]{amsart}
\usepackage{amsmath, amsthm, amscd, amsfonts, amssymb, mathtools}
\usepackage{color, geometry, relsize, hyperref}
\usepackage[dvipsnames]{xcolor}
\newtheorem{theorem}{Theorem}[section]
\newtheorem{lemma}[theorem]{Lemma}

\newtheorem{corollary}[theorem]{Corollary}

\theoremstyle{definition}

\newtheorem{example}[theorem]{Example}

\theoremstyle{remark}

\numberwithin{equation}{section}

\begin{document}
\setcounter{page}{1}

\title[Eigenvalue bounds]
{Bounds on the moduli of eigenvalues of rational matrices}
	
\author[Pallavi, Shrinath and Sachindranath]{Pallavi Basavaraju, Shrinath Hadimani 
and Sachindranath Jayaraman}
\address{School of Mathematics\\ 
Indian Institute of Science Education and Research Thiruvananthapuram\\ 
Maruthamala P.O., Vithura, Thiruvananthapuram -- 695 551, Kerala, India.}
\email{(pallavipoorna20, srinathsh3320, sachindranathj)@iisertvm.ac.in, 
sachindranathj@gmail.com}

	
	
\begin{abstract} 
A rational matrix is a matrix-valued function 
$R(\lambda): \mathbb{C} \rightarrow M_p$ such that $R(\lambda) = \begin{bmatrix}
r_{ij}(\lambda)
\end{bmatrix} _{p\times p}$, where $r_{ij}(\lambda)$ are scalar complex rational 
functions in $\lambda$ for $i,j=1,2,\ldots,p$. The aim of this paper is to 
obtain bounds on the moduli of eigenvalues of rational matrices in terms of 
the moduli of their poles. To a given rational matrix $R(\lambda)$ we associate a 
block matrix $\mathcal{C}_R$ whose blocks consist of the coefficient matrices of 
$R(\lambda)$, as well as a scalar real rational function $q(x)$ whose coefficients 
consist of the norm of the coefficient matrices of $R(\lambda)$. We prove 
that a zero of $q(x)$ which is greater than the moduli 
of all the poles of $R(\lambda)$ will be an upper bound on the moduli of 
eigenvalues of $R(\lambda)$. Moreover, by using a block matrix associated with 
$q(x)$, we establish bounds on the zeros of $q(x)$, which in turn yields  
bounds on the moduli of eigenvalues of $R(\lambda)$. 
\end{abstract}
	
\maketitle
	
\section{Introduction}\label{sec-1}

The general expression for a nonlinear matrix-valued function 
is $G(\lambda) \colon \mathbb{C} \to M_p$ such that $G(\lambda) = \displaystyle 
\sum_{i=0}^{m} A_i f_i(\lambda)$, where $M_p$ is the set of all $p \times p$ complex 
matrices, $f_i \colon \mathbb{C} \to \mathbb{C}$ are nonlinear functions of 
$\lambda$, and $A_i \in M_p$ for $i=0,1,2,\ldots,m$. A nonlinear eigenvalue 
problem, abbreviated as NEP, is to find a scalar $\lambda_0 \in \mathbb{C}$ 
and a nonzero vector $v \in \mathbb{C}^n$ such that $G(\lambda_0)v = 0$. The 
scalar $\lambda_0$ is called an eigenvalue of $G(\lambda)$ and the nonzero 
vector $v$ is called an eigenvector corresponding to the eigenvalue $\lambda_0$. 
In particular, if $f_i(\lambda) = \lambda^i$ for $i=0,1,2,\ldots,m$, then 
$G(\lambda)$ is called a $p \times p$ matrix polynomial of degree $m$ and is 
denoted by $P(\lambda)$. 
In this case the corresponding NEP is called a polynomial eigenvalue 
problem (PEP). When the $f_i(\lambda)$ are scalar complex rational functions for 
$i=0,1,2,\ldots,m$, $G(\lambda)$ is called a $p\times p$ rational matrix and 
denoted by $R(\lambda)$. The corresponding eigenvalue problem is called a rational 
eigenvalue problem (REP). Note that rational matrices are the matrices whose 
entries are scalar rational functions. In other words, a $p\times p$ rational 
matrix is a matrix-valued function $R(\lambda)\colon \mathbb{C} \to M_p$ such 
that $R(\lambda) = \begin{bmatrix}
r_{ij}(\lambda)
\end{bmatrix} _{p\times p}$, where $r_{ij}(\lambda)$ are scalar complex rational 
functions in $\lambda$ for $i,j=1,2,\ldots,p$. Rational matrices are sometimes 
called matrix rational functions (see for instance \cite{Su-Bai}). Throughout 
this manuscript we represent an arbitrary nonlinear matrix-valued function by 
$G(\lambda)$, a rational matrix by $R(\lambda)$ and a matrix polynomial by 
$P(\lambda)$. The following are examples of NEPs which can be found in 
\cite{Saad-Guide-Meidlar}, \cite{Su-Bai} 
and \cite{Betcke-Higham-Tisseur} respectively.
\begin{itemize}
\item[(i)] $G(\lambda)v= 0$, where $G(\lambda)$ is the nonlinear matrix-valued 
function given by $ G(\lambda) = -B_0 + I \lambda + A_1 e ^ {-\lambda \tau}$, 
where $B_0 = \begin{bmatrix}
-5 & 1 \\
2 & -6
\end{bmatrix}, A_1 = -\begin{bmatrix}
-2 & 1 \\
4 & -1
\end{bmatrix}$ and $\tau = 1$.
\item[(ii)] $R(\lambda)v=0$, where $R(\lambda)$ is the rational matrix given by 
$R(\lambda) = A-B\lambda + \displaystyle \sum_{i = 1}^{m} E_i \frac{\lambda}{\lambda-\sigma_i}$, 
where $\sigma_i$ are positive, $A $ and $B$ are symmetric positive definite 
matrices, and $E_i= C_iC_i^T$, $C_i \in \mathbb{R}^{n\times r}$ has rank $r_i$ 
for $ i=1,2,\ldots,m$.
\item[(iii)] $P(\lambda)v=0$, where $P(\lambda)$ is the matrix polynomial given 
by $P(\lambda) = I \lambda^2 +B \lambda +C$, where $B = \begin{bmatrix}
0 & 1+\alpha \\
1 & 0
\end{bmatrix}$, $C = \begin{bmatrix}
1/2 & 0\\
0 & 1/4
\end{bmatrix}$ and $\alpha \geq 0$.
\end{itemize} 
NEPs emerge in a variety of applications. To name a few, they emerge in the description of eigenvibration of a string with a load of mass attached by an elastic string 
in \cite{Betcke-Higham-Tisseur} and in optimization of acoustic emissions of high 
speed trains in \cite{Mackey}. NEPs also arise in application to electronic structure of calculations of quantum 
dots in \cite{Heinrich} and in application to photonic crystals in 
\cite{Engstrom-Langer-Tretter}. The origins of the study of NEPs lie in linear systems and control theory with the investigation of structural properties such as the location of poles and eigenvalues, see \cite{Kailath}. 
A comprehensive treatment of nonlinear eigenvalue problems can be found in 
\cite{Dopico}, \cite{Mehrmann-Voss} 
and the references cited therein. Many nonlinear eigenvalue problems encountered 
in practical applications are either rational matrices or can be approximated by 
a rational matrix using rational approximation methods \cite{Saad-Guide-Meidlar}. 
For a rational matrix $R(\lambda)$, eigenvalues can be approximated using recursive methods, as outlined in \cite{Mehrmann-Voss}, where one starts 
with an initial guess for the eigenvalues. The initial guess plays an important role for the speed of convergence of the iterative approximation. There is, however, a 
different method where one goes through a rational approximation 
\cite{Saad-Guide-Meidlar} , where the eigenvalues are computed inside a closed 
contour. In either case it is necessary to determine the location of eigenvalues to reduce computational effort.

In \cite{Hadimani-Pallavi-Jayaraman}, the authors consider a 
particular type of rational matrices and derive bounds on the moduli of 
eigenvalues. The current work is an extension of \cite{Hadimani-Pallavi-Jayaraman} 
for an arbitrary rational matrix. We first show that an arbitrary rational matrix 
$R(\lambda)= \begin{bmatrix}
r_{ij}(\lambda)
\end{bmatrix} _{p\times p}$ can be expressed as $R(\lambda) = P(\lambda) +
\displaystyle \sum_{j=1}^{n} \sum_{k=1}^{m_j} \frac{B_k^{(j)}}{(\lambda-a_j)^k},$
where $P(\lambda)=C_m\lambda^m- C_{m-1}\lambda^{m-1} - \cdots -C_1 \lambda - C_0$ 
is a matrix polynomial and the coefficients $C_i$'s, $B_k^{(j)}$'s are square matrices of size 
$p\times p$ and $a_j \in \mathbb{C}$. We associate to $R(\lambda)$ a block matrix 
$\mathcal{C}_R$ whose blocks consist of the coefficient matrices of $R(\lambda)$. 
We then associate a scalar real rational function $q(x) = x^m-||C_{m-1}||x^{m-1}-
\cdots - ||C_1|| x - ||C_0|| - \displaystyle \sum_{j=1}^{n} \sum_{k=1}^{m_j} 
\frac{||B_k^{(j)}||}{(x- |a_j|)^k}$ to $R(\lambda)$, where $||\cdot ||$ is any 
induced matrix norm. We prove that a zero of the real rational function $q(x)$ 
which is greater than the moduli of all the poles of $R(\lambda)$ serves as an 
upper bound on the moduli of eigenvalues of $R(\lambda)$. Since $q(x)$ can 
be considered as a rational matrix of size $1 \times 1$, we use bounds on the 
numerical radius of the block matrix associated with the scalar real rational 
function $q(x)$ to derive bounds on the moduli of eigenvalues of the rational 
matrix $R(\lambda)$. In Section \ref{sec-4} we have considered certain examples 
and compared our bounds with the existing results from the literature. In these 
examples we have observed that our bounds are better. To the best of our 
knowledge there are no ready bounds in the literature on the moduli of 
eigenvalues of rational matrices. However, one can convert a given rational 
matrix into a matrix polynomial and use existing results to obtain bounds. 
The emerging matrix polynomial frequently has high degree and coefficient matrices involve many 
terms, which results in tedious computations. 

The manuscript is organized as follows. Section \ref{sec-1} is introductory. 
Section \ref{sec-2} collects basic facts about rational matrices and inequalities 
involving the numerical radius that will be used later on. The main results 
are presented in Section \ref{sec-3}, which is subdivided into subsections for 
ease of reading. These subsections are self-explanatory. The manuscript ends 
with numerical examples illustrating the results obtained.

\section{Preliminaries}\hspace*{\fill}
\label{sec-2}

The space of all $p \times p$ matrices over complex numbers is denoted by $M_{p}$. 
For $A \in M_p$, the expressions $||A||_2, ||A||_\infty$ and $||A||_1$ stand for the spectral norm, the maximum row sum norm and the maximum column 
sum norm of $A$, respectively. A $p\times p$ rational matrix $R(\lambda)$ is 
said to be regular if the determinant of $R(\lambda)$ does not vanish identically. 
For a regular rational matrix $R(\lambda)$ a scalar $\lambda_0$ is called an 
eigenvalue of $R(\lambda)$ if there exists a nonzero vector $v \in \mathbb{C}^n$ 
such that $R(\lambda_0)v=0$, with each entry of $R(\lambda_0)$ to be finite complex 
number; that is, entries of $R(\lambda_0)$ cannot be undefined or infinite. 
The vector $v$ obtained is called an eigenvector of $R(\lambda)$ corresponding 
to the eigenvalue $\lambda_0$.

\begin{example}
Consider $R_1(\lambda)= \begin{bmatrix}
\frac{1}{\lambda} & 1 \\
1 & \lambda
\end{bmatrix}$. Then $\text{det}R_1(\lambda)= 
(\frac{1}{\lambda}) \lambda -1 =0$. Therefore $R_1(\lambda)$ is not regular. On 
the other hand let $R_2(\lambda)= \begin{bmatrix}
\lambda & 0 \\
0 & \lambda
\end{bmatrix}$. The determinant of $R_2(\lambda)$ is $\lambda^2$, which is not 
identically zero. Hence $R_2(\lambda)$ is regular. Note that for any nonzero vector 
$v =\begin{bmatrix}
v_1 \\
v_2
\end{bmatrix} \in \mathbb{C}^2$, $R_2(\lambda)v=0$ if and only if $\lambda = 0$. 
Therefore $\lambda = 0$ is the only eigenvalue of $R_2(\lambda)$. Further, if the 
regularity condition is relaxed then the set of moduli of eigenvalues of a rational 
matrix need not be bounded. To see this, if $\lambda_0$ is any nonzero complex number 
then the vector $v =\begin{bmatrix}
-\lambda_0 \\
1
\end{bmatrix} \in \mathbb{C}^2$ satisfies $R_1(\lambda_0)v= \begin{bmatrix}
\frac{1}{\lambda_0} & 1 \\
1 & \lambda_0
\end{bmatrix} \begin{bmatrix}
-\lambda_0 \\
1
\end{bmatrix} = \begin{bmatrix}
0 \\
0
\end{bmatrix}$. Therefore every nonzero complex number is an eigenvalue of $R_1(\lambda)$,
which in turn proves that the set of moduli of eigenvalues of $R_1(\lambda)$ is 
unbounded.
\end{example} 

\medskip
A $p \times p$ matrix polynomial $P(\lambda)$ is said to be unimodular 
if $\text{det}(P(\lambda))$ is a nonzero constant (see page 230 of \cite{Gohberg} for details). Consider 
a $p\times p$ rational matrix $R(\lambda)=[r_{ij}(\lambda)]$,  where $\displaystyle r_{ij}(\lambda) = 
\frac{p_{ij}(\lambda)}{q_{ij}(\lambda)}$ is a scalar complex rational function and 
$p_{ij}, q_{ij}$ are scalar complex polynomials for $i,j=1,2,\ldots,p$. Let 
$d(\lambda)$ be the monic least common multiple of denominators of all the entries 
of $R(\lambda)$. There exist unimodular matrix polynomials $U(\lambda)$ and 
$V(\lambda)$ of size $p\times p$ such that 
\begin{equation} \label{Eqn-Smith-McMillan form}
U(\lambda) R(\lambda) V(\lambda) = \begin{bmatrix}
\displaystyle \text{diag}\bigg(\frac{\epsilon_1(\lambda)}{\psi_1(\lambda)},\cdots,
\frac{\epsilon_r(\lambda)}{\psi_r(\lambda)}\bigg) & 0 \\
0 & 0
\end{bmatrix}_{p\times p},
\end{equation}
where $r$ is the size of the largest nonvanishing minor of $R(\lambda)$ such that 
$r \leq p$, $\epsilon_i(\lambda), \psi_i(\lambda)$ are monic scalar polynomials 
with $\gcd\textbf{(}\epsilon_i(\lambda), \psi_i(\lambda)\textbf{)} =1$; that is, 
$ \epsilon_i(\lambda)$ and $\psi_i(\lambda)$ are coprime and $\epsilon_i(\lambda)$ 
divides $\epsilon_{i+1}(\lambda)$, $\psi_{i+1}(\lambda)$ divides $ \psi_i(\lambda)$ 
for $1 \leq i \leq r-1$, and $\psi_1(\lambda) = d(\lambda)$ (see \cite{Rosenbrock}, 
\cite{Amparan-Dopico} for details). The rational matrix on the right hand side of 
Equation \eqref{Eqn-Smith-McMillan form} is called the Smith-McMillan form of 
$R(\lambda)$. A scalar $\lambda_0 \in \mathbb{C}$ is said to be a finite pole of 
$R(\lambda)$ if $\lambda_0$ is a zero of $\psi_1(\lambda)$. If $\lambda_0$ is a 
finite pole of $R(\lambda)$, we can write 
$\psi_i(\lambda) = (\lambda - \lambda_0)^{\alpha_i} \widetilde{\psi_i}(\lambda)$, 
where $\alpha_i$ is a nonnegative integer and $\widetilde{\psi_i}(\lambda_0) \neq 0$ 
for $1\leq i \leq r$. The nonzero numbers in $(\alpha_1,\alpha_2,\cdots,\alpha_r)$ 
are called the partial multiplicities of $\lambda_0$ as a pole of $R(\lambda)$.  
The maximum partial multiplicity of $\lambda_0$ is called the order 
of $\lambda_0$ as a pole of $R(\lambda)$. In other words $\lambda_0$ is a pole of 
$R(\lambda)$ of order $k$ if it is a zero of $\psi_1(\lambda)$ of order $k$.

Recall that a complex scalar rational function $r(\lambda)$ can 
be expressed in the form, 
\begin{equation}\label{Eqn-scalar rational function}
\displaystyle r(\lambda) = p(\lambda) +  \frac{s(\lambda)}{q(\lambda)} 	
\end{equation}
where $p(\lambda)$, $s(\lambda)$ and $q(\lambda)$ are polynomials of degree $m$, 
$d_1$ and $d_2$ respectively with $d_1<d_2$, $q(\lambda) \not \equiv 0$ and 
$\gcd\textbf{(}s(\lambda), q(\lambda)\textbf{)} =1$. Let $q(\lambda) = 
c(\lambda-a_1)^{m_1} (\lambda-a_2)^{m_2} \cdots (\lambda-a_n)^{m_n}$ 
be the factorization of $q(\lambda)$, where $a_j$ is a zero of $q(\lambda)$ of 
multiplicity $m_j \in \mathbb{N}$ for $j=1,2,\ldots,n$ and $n$ is the number of 
distinct zeros of $q(\lambda)$. It is known that $\displaystyle 
\frac{s(\lambda)}{q(\lambda)}$ can be expressed as a polynomial in $\displaystyle 
\frac{1}{\lambda-a_j}$ for $j =1, \ldots, n$ \cite{Ponnusamy}. 
Therefore $r(\lambda)$ can be written as,
\begin{equation}\label{Eqn-scalar rational function-general form}
r(\lambda) = p(\lambda) + \displaystyle \sum_{j=1}^{n} \sum_{k=1}^{m_j} 
\frac{b_k^{(j)}}{(\lambda-a_j)^k}
\end{equation}
where, $a_j$'s denote the poles of $r(\lambda)$ with orders $m_j \geq 1$ for 
$1 \leq j \leq n$. Thus, given an arbitrary $p \times p$ rational 
matrix $R(\lambda)$ we can write 
\begin{equation}\label{Eqn- rational matrix}
R(\lambda) = P(\lambda) +\displaystyle \sum_{j=1}^{n} \sum_{k=1}^{m_j} 
\frac{B_k^{(j)}}{(\lambda-a_j)^k},
\end{equation}
where $P(\lambda)=C_m\lambda^m- C_{m-1}\lambda^{m-1} - \cdots -C_1 \lambda - C_0$ 
is a matrix polynomial, $C_i$'s, $B_k^{(j)}$'s are square matrices of size 
$p\times p$ and each $a_j \in \mathbb{C}$ is a finite pole of $R(\lambda)$ of order 
$m_j$ for $1 \leq j \leq n$. The particular case of \eqref{Eqn- rational matrix} 
when $P(\lambda)$ is a monic linear matrix polynomial and each pole is of order 
$1$ reduces to the following form studied in \cite{Hadimani-Pallavi-Jayaraman} in 
the context of eigenvalue location  
\begin{equation}\label{Eqn-rational matrix-linear}
R(\lambda) = \displaystyle -C_0 +I\lambda +\frac{B_1}{\lambda-a_1}+ 
\frac{B_2}{\lambda-a_2} + \cdots + \frac{B_n} {\lambda-a_n},
\end{equation} 
where $a_j$'s are distinct complex numbers. 

\noindent 
\textbf{Assumption:} Throughout this paper we assume that the degree of 
$P(\lambda)$ is at least one and that the leading coefficient of $P(\lambda)$ is 
the identity matrix of size $p$; that is, $C_m=I_p$. 

\medskip
As mentioned in the introduction, we make use of bounds on the numerical radius of a matrix 
to estimate bounds on the moduli of eigenvalues of $R(\lambda)$. Given $A \in M_p$, the set 
$W(A) = \{ x^*Ax : x \in \mathbb{C}^p$ and  $x^*x=1\}$ is called the numerical range of 
$A$ and $w(A) = \sup \{ |\lambda| : \lambda \in W(A) \}$ is called the numerical radius 
of $A$. If $\rho(A)$ denotes the spectral radius of $A$, then for any eigenvalue $\lambda_0$ 
of $A$, 
\begin{equation}\label{numerical radius-relationship}
|\lambda_0| \leq \rho(A) \leq w(A) \leq ||A||_2.
\end{equation}
Other well known properties of the numerical radius are  
\begin{itemize}
\item $w(A) \geq 0$.
\item $w(\alpha A) = |\alpha| w(A)$ for every $\alpha \in \mathbb{C}$.
\item $w(A+B) \leq w(A)+w(B)$ for $A$, $B$ in $M_p$. 
\end{itemize}

\medskip
Before stating our main results we state without proof bounds on the numerical radius 
of matrices.

\begin{lemma}[\cite{Abu-Omar}] \label{Lem-numerical radius-1}
Let $A \in M_k$, $B \in M_{k,s}$, $C \in M_{s,k}$ and $D \in M_{s}$, and 
let $K = \begin{bmatrix}
A & B \\
C & D
\end{bmatrix}$. Then 
$w(K) \leq \displaystyle \frac{1}{2} \Big(w(A)+ w(D) +\sqrt{(w(A)-w(D))^2+ 4w^2(K_0)}\Big)$,
where $K_0 = \begin{bmatrix}
0 & B \\
C & 0
\end{bmatrix}$. 
\end{lemma}

\begin{lemma}[\cite{Yamazaki}]\label{Lem-numerical radius-2}
Let $B \in M_{k,s}$, $C \in M_{s,k}$, then
$\displaystyle w\left(\begin{bmatrix}
0 & B \\
C & 0
\end{bmatrix}\right) \leq \frac{1}{2}\big(||B||_2+ ||C||_2 \big)$.
\end{lemma}

\medskip
\noindent
As a consequence of Lemmas \ref{Lem-numerical radius-1} and 
\ref{Lem-numerical radius-2} we have the following inequality:
\begin{equation}\label{Eqn-numerical radius inq}
w(K) \leq \displaystyle \frac{1}{2} \Big(w(A)+ w(D) +\sqrt{(w(A)-w(D))^2+ 
(||B||_2+||C||_2)^2} \Big)
\end{equation}

\noindent
It is also not hard to prove the following result as in the example given in Section 
$1.3$ of \cite{Gustafson-Rao}. We state this as a lemma for ready reference.

\begin{lemma}\label{Lem-numerical radius-Jordan block}
Let $L_n$ be $n \times n$ matrix given by $L_n =\begin{bmatrix}
a & 0 & \cdots & 0 & 0\\
1 & a & \cdots & 0 & 0\\
\vdots & \vdots & \ddots & \vdots & \vdots \\
0 & 0 & \cdots & 1 & a 
\end{bmatrix}$, where $a$ is a complex number. Then $w(L_n) \leq \displaystyle |a|+ 
\cos\Big(\frac{\pi}{n+1}\Big)$. 
If $a$ is a nonnegative real number then equality holds.
\end{lemma}

\section{Main results}\label{sec-3}
This section contains the main results of this manuscript.
\subsection{Associated block matrix corresponding to rational matrix}\label{sec-3.1}

\

Consider the rational matrix $R(\lambda)$ given in Equation 
\eqref{Eqn- rational matrix}. We associate to $R(\lambda)$ a block 
matrix $\mathcal{C}_R$ of size \small{$p\left(m + \displaystyle \sum_{j=1}^{n}
\frac{m_j(m_j+1)}{2}\right) \times p\left(m + \displaystyle \sum_{j=1}^{n}
\frac{m_j(m_j+1)}{2}\right)$}, which is defined as,
\begin{equation}\label{Eqn-Block matrix}
\mathcal{C}_R= \begin{bmatrix}
\mathcal{A}_n & 0 & \cdots & 0 & -\mathcal{F}_n \\
0 & \mathcal{A}_{n-1} & \cdots & 0 & -\mathcal{F}_{n-1} \\
\vdots & \vdots & \ddots & \vdots & \vdots \\
0 & 0 & \cdots & \mathcal{A}_1 & -\mathcal{F}_1 \\
\mathcal{B}_n & \mathcal{B}_{n-1} & \cdots & \mathcal{B}_1 & \mathcal{B}_0
\end{bmatrix},
\end{equation}
where for each $1 \leq j \leq n$ the blocks are defined as follows: \\
$\mathcal{A}_j = \text{diag}(\textbf{A}_{m_j}^{(j)},\textbf{A}_{m_j-1}^{(j)},
\cdots,\textbf{A}_1^{(j)})$ is a block diagonal matrix of size $\displaystyle 
\frac{pm_j(m_j+1)}{2} \times \frac{pm_j(m_j+1)}{2}$, with $\textbf{A}_k^{(j)} 
=\begin{bmatrix}
a_jI & I & 0 & \cdots & 0 & 0 \\
0 & a_jI & I & \cdots & 0 & 0 \\
\vdots & \vdots & \vdots & \ddots & \vdots & \vdots \\
0 & 0 & 0 & \cdots & a_jI & I \\
0 & 0 & 0 & \cdots & 0 & a_jI
\end{bmatrix}_{pk \times pk}$, $1\leq k \leq m_j$ and 

\noindent
$\mathcal{B}_j = \begin{bmatrix}
\textbf{B}_{m_j}^{(j)} & \textbf{B}_{m_j-1}^{(j)} & \cdots & \textbf{B}_1^{(j)}
\end{bmatrix}$ is a block row matrix of size $\displaystyle pm \times 
\frac{pm_j(m_j+1)}{2}$, with $\textbf{B}_k^{(j)} = \begin{bmatrix}
0 & 0 & \cdots & 0 \\
0 & 0 & \cdots & 0 \\
\vdots & \vdots & \ddots & \vdots \\
B_{k}^{(j)} & 0 & \cdots & 0
\end{bmatrix}_{pm \times pk}, 1 \leq k \leq m_j $, \\
$\mathcal{B}_0 = \begin{bmatrix}
0 & I & 0 & \cdots & 0\\
0 & 0 & I & \cdots & 0\\
\vdots & \vdots & \vdots & \ddots & \vdots\\
0 & 0 & 0 & \cdots & I\\
C_0 & C_1 & C_2 & \cdots & C_{m-1}\\ 
\end{bmatrix}_{pm \times pm}$ and
$\mathcal{F}_j = \begin{bmatrix}
\textbf{F}_{m_j} \\
\textbf{F}_{m_j-1} \\
\vdots \\
\textbf{F}_1
\end{bmatrix}$ is block column \\
matrix of size $\displaystyle \frac{pm_j(m_j+1)}{2} 
\times pm$ with $ \textbf{F}_k =\begin{bmatrix}
0 & 0 & \cdots & 0\\
0 & 0 & \cdots & 0\\
\vdots & \vdots & \ddots & \vdots \\
I & 0 & \cdots & 0 \\
\end{bmatrix}_{pk \times pm}, 1 \leq k \leq m_j$ for $j = 1,\cdots, n$. 

\medskip
For illustration we present the block matrix associated with a 
scalar complex rational function. Consider a scalar complex rational function 
given by
\begin{equation}\label{Eqn-scalar rational function-2}
r(\lambda) = \lambda^m- c_{m-1}\lambda^{m-1} - \cdots -c_1 \lambda - c_0 -  
\displaystyle \sum_{j=1}^{n} \sum_{k=1}^{m_j} \frac{b_k^{(j)}}{(\lambda-a_j)^k}, 
\end{equation}
where $b_k^{(j)}$'s, $a_j$'s and $c_i$'s are complex numbers.
Since $r(\lambda)$ can be considered as a $1\times 1$ rational matrix, 
we have an associated block matrix of size $\left( m + \displaystyle 
\sum_{j=1}^{n}\frac{m_j(m_j+1)}{2}\right) \times \left(m + \displaystyle 
\sum_{j=1}^{n}\frac{m_j(m_j+1)}{2}\right)$ given by 
\begin{equation}\label{Eqn-block matrix for scalar rational fcn}
\mathcal{C}_r= \begin{bmatrix}
\mathcal{A}_n & 0 & \cdots & 0 & -\mathcal{F}_n \\
0 & \mathcal{A}_{n-1} & \cdots & 0 & -\mathcal{F}_{n-1} \\
\vdots & \vdots & \ddots & \vdots & \vdots \\
0 & 0 & \cdots & \mathcal{A}_1 & -\mathcal{F}_1 \\
\mathcal{B}_n & \mathcal{B}_{n-1} & \cdots & \mathcal{B}_1 & \mathcal{B}_0
\end{bmatrix}, 
\end{equation}
where for each $1 \leq j \leq n$,   
$\mathcal{A}_j = \text{diag}(A_{m_j}^{(j)},A_{m_j-1}^{(j)},\cdots,A_1^{(j)})$ is a 
block diagonal matrix of size $\displaystyle \frac{m_j(m_j+1)}{2} \times \frac{m_j(m_j+1)}{2}$ with $A_k^{(j)} 
=\begin{bmatrix}
a_j & 1 & 0 & \cdots & 0 & 0 \\
0 & a_j & 1 & \cdots & 0 & 0 \\
\vdots & \vdots & \vdots & \ddots & \vdots & \vdots \\
0 & 0 & 0 & \cdots & a_j & 1 \\
0 & 0 & 0 & \cdots & 0 & a_j
\end{bmatrix}_{k \times k}$ for $1\leq k \leq m_j$,
$\mathcal{B}_j = \begin{bmatrix}
B_{m_j}^{(j)} & B_{m_j-1}^{(j)} & \cdots & B_1^{(j)}
\end{bmatrix}$ is a block row matrix of size $\displaystyle m \times 
\frac{m_j(m_j+1)}{2}$ with $B_k^{(j)} = \begin{bmatrix}
0 & 0 & \cdots & 0 \\
0 & 0 & \cdots & 0 \\
\vdots & \vdots & \ddots & \vdots \\
-b_{k}^{(j)} & 0 & \cdots & 0
\end{bmatrix}_{m \times k}$ for $ 1 \leq k \leq m_j $, 
$\mathcal{B}_0 = \begin{bmatrix}
0 & 1 & 0 & \cdots & 0\\
0 & 0 & 1 & \cdots & 0\\
\vdots & \vdots & \vdots & \ddots & \vdots\\
0 & 0 & 0 & \cdots & 1\\
c_0 & c_1 & c_2 & \cdots & c_{m-1}\\ 
\end{bmatrix}_{m \times m}$, and
$\mathcal{F}_j = \begin{bmatrix}
F_{m_j} \\
F_{m_j-1} \\
\vdots \\
F_1
\end{bmatrix}$ is a block column matrix of size $\displaystyle \frac{m_j(m_j+1)}{2} 
\times m$ with $ F_k =\begin{bmatrix}
0 & 0 & \cdots & 0\\
0 & 0 & \cdots & 0\\
\vdots & \vdots & \ddots & \vdots \\
1 & 0 & \cdots & 0 \\
\end{bmatrix}_{k \times m}, 1 \leq k \leq m_j$ for $j = 1,\cdots, n$. 
Note that for $1 \leq j \leq n$, the blocks $A_k^{(j)}$'s in the matrix $\mathcal{A}_j$ are Jordan like blocks of size $k \times k$ corresponding to the pole $a_j$, where $1 \leq k \leq m_j$.

\medskip
Our first result is that the set of eigenvalues of $R(\lambda)$ is contained 
in the set of eigenvalues of $\mathcal{C}_R$.

\begin{theorem}\label{Thm- eigenvalues- matrix rational function}
Let $R(\lambda)$ be a $p\times p$ rational matrix given in Equation 
\eqref{Eqn- rational matrix} and $\mathcal{C}_R$ be the associated block 
matrix described in Equation \eqref{Eqn-Block matrix}. Then the eigenvalues 
of $R(\lambda)$ are also eigenvalues of the corresponding block matrix
$\mathcal{C}_R$.
\end{theorem}

\begin{proof}
Let $(\lambda_0,v)$ be an eigenpair of $R(\lambda)$ so that $R(\lambda_0)v = 0$.  \\
For $1 \leq j \leq n$, define a vector 
\[
\mathcal{Y}_j = \begin{bmatrix}
    \textbf{Y}_{m_j}^{(j)} \\
    \textbf{Y}_{m_j -1}^{(j)} \\
    \vdots \\
    \textbf{Y}_{2}^{(j)} \\
    \textbf{Y}_{1}^{(j)}
\end{bmatrix},
\]
where 
\[
\textbf{Y}_k^{(j)} = \begin{bmatrix}
    \frac{-1}{(\lambda_0-a_j)^k} v\\
    \frac{-1}{(\lambda_0-a_j)^{k-1}} v\\
    \vdots\\
    \frac{-1}{(\lambda_0-a_j)^2}v\\
    \frac{-1}{(\lambda_0-a_j)}v
\end{bmatrix}
\]
for $1 \leq k \leq m_j$ and 
\[
\mathcal{Y}_0 = \begin{bmatrix}
    v\\
    \lambda_0v\\
    \vdots\\
    \lambda_0^{m-1}v
\end{bmatrix}.
\]

\noindent
Let 
\[
\mathcal{Y} = \begin{bmatrix}
    \mathcal{Y}_n \\
    \mathcal{Y}_{n-1} \\
    \vdots \\  
    \mathcal{Y}_1 \\
    \mathcal{Y}_0
\end{bmatrix}.
\]
\\                                                                                           
Consider 
\[
\mathcal{C}_R \mathcal{Y} = \begin{bmatrix}
    \mathcal{A}_n & 0 & \cdots & 0 & -\mathcal{F}_n \\
    0 & \mathcal{A}_{n-1} & \cdots & 0 & -\mathcal{F}_{n-1} \\
    \vdots & \vdots & \ddots & \vdots & \vdots \\
    0 & 0 & \cdots & \mathcal{A}_1 & -\mathcal{F}_1 \\
    \mathcal{B}_n & \mathcal{B}_{n-1} & \cdots & \mathcal{B}_1 & \mathcal{B}_0
\end{bmatrix} \begin{bmatrix}
    \mathcal{Y}_n \\
    \mathcal{Y}_{n-1} \\
    \vdots \\
    \mathcal{Y}_1 \\
    \mathcal{Y}_0
\end{bmatrix} =
\begin{bmatrix} 
    \mathcal{A}_n \mathcal{Y}_n-\mathcal{F}_n\mathcal{Y}_0 \\
    \mathcal{A}_{n-1} \mathcal{Y}_{n-1}-\mathcal{F}_{n-1}\mathcal{Y}_0 \\
    \vdots \\
    \mathcal{A}_1 \mathcal{Y}_1-\mathcal{F}_1\mathcal{Y}_0 \\
    \sum_{j=1}^{n} \mathcal{B}_j \mathcal{Y}_j + \mathcal{B}_0\mathcal{Y}_0
\end{bmatrix}.
\]\\
For $1 \leq j \leq n$, we have 
\begin{align*}
\mathcal{A}_{j} \mathcal{Y}_j-\mathcal{F}_{j}\mathcal{Y}_0= &\begin{bmatrix}
    \textbf{A}_{m_j}^{(j)} & 0 & \cdots & 0 & 0 \\
    0 & \textbf{A}_{m_j -1}^{(j)} & \cdots & 0 & 0 \\
    \vdots & \vdots & \ddots & \vdots & \vdots \\
    0 & 0 & \cdots & \textbf{A}_{2}^{(j)} & 0 \\
    0 & 0 & \cdots & 0 & \textbf{A}_{1}^{(j)} \\
\end{bmatrix} \begin{bmatrix}
    \textbf{Y}_{m_j}^{(j)} \\
    \textbf{Y}_{m_j -1}^{(j)} \\
    \vdots \\
    \textbf{Y}_{2}^{(j)} \\
    \textbf{Y}_{1}^{(j)}\\
\end{bmatrix}- \begin{bmatrix}
    \textbf{F}_{m_j} \\
    \textbf{F}_{m_j-1} \\
    \vdots \\
    \textbf{F}_2\\
    \textbf{F}_1
\end{bmatrix}\mathcal{Y}_0\\ 
=& \begin{bmatrix}
    \textbf{A}_{m_j}^{(j)}\textbf{Y}_{m_j}^{(j)}-\textbf{F}_{m_j}\mathcal{Y}_0\\
    \textbf{A}_{m_j-1}^{(j)}\textbf{Y}_{m_j-1}^{(j)}-\textbf{F}_{m_j-1}\mathcal{Y}_0\\
    \vdots\\
    \textbf{A}_2^{(j)}\textbf{Y}_2^{(j)}-\textbf{F}_2\mathcal{Y}_0\\
    \textbf{A}_1^{(j)}\textbf{Y}_1^{(j)}-\textbf{F}_1\mathcal{Y}_0
\end{bmatrix}.
\end{align*} \\
For $1 \leq k \leq m_j$, consider 
\begin{align*}
& \textbf{A}_k^{(j)}\textbf{Y}_k^{(j)}-\textbf{F}_k\mathcal{Y}_0 \nonumber\\
& = \begin{bmatrix}
    a_jI & I & \cdots & 0 & 0 \\
    0 & a_jI  & \cdots & 0 & 0 \\
    \vdots & \vdots & \ddots & \vdots & \vdots \\
    0 & 0 & \cdots & a_jI & I \\
    0 & 0 & \cdots & 0 & a_jI
\end{bmatrix} \begin{bmatrix}
    \frac{-1}{(\lambda_0-a_j)^k} v\\
    \frac{-1}{(\lambda_0-a_j)^{k-1}} v\\
    \vdots\\
    \frac{-1}{(\lambda_0-a_j)^2}v\\
    \frac{-1}{(\lambda_0-a_j)}v
\end{bmatrix}-\begin{bmatrix}
    0 & 0 & \cdots & 0\\
    0 & 0 & \cdots & 0\\
    \vdots & \vdots & \ddots & \vdots \\
    0 & 0 & \cdots &  0 \\
    I & 0 & \cdots & 0 \\
\end{bmatrix} \begin{bmatrix}
    v\\
    \lambda_0v\\
    \vdots\\
    \lambda_0^{m-2}v \\
    \lambda_0^{m-1}v\\
\end{bmatrix} \\ 
& = \lambda_0 \begin{bmatrix}
    \frac{-1}{(\lambda_0-a_j)^k} v\\
    \frac{-1}{(\lambda_0-a_j)^{k-1}} v\\
    \vdots\\
    \frac{-1}{(\lambda_0-a_j)^2}v\\
    \frac{-1}{(\lambda_0-a_j)}v
\end{bmatrix} = \lambda_0 \textbf{Y}_k^{(j)}. 
\end{align*}
This implies 
\begin{equation}\label{Eqn-a}
\mathcal{A}_{j} \mathcal{Y}_j-\mathcal{F}_{j}\mathcal{Y}_0  =\lambda_0 \mathcal{Y}_j  \ \text{for} \ j= 1, 2, \cdots, n. 
\end{equation}\\
Now consider,
\begin{align}
 \displaystyle \sum_{j=1}^{n} \mathcal{B}_j \mathcal{Y}_j+ \mathcal{B}_0\mathcal{Y}_0  
& = \sum_{j=1}^{n} \begin{bmatrix}
    \textbf{B}_{m_j}^{(j)} & \textbf{B}_{m_j -1}^{(j)} & \cdots & \textbf{B}_1^{(j)}  
\end{bmatrix} \begin{bmatrix}
    \textbf{Y}_{m_j}^{(j)}\\
    \textbf{Y}_{m_j -1}^{(j)}\\
    \vdots \\
    \textbf{Y}_1^{(j)}\\
\end{bmatrix}+\mathcal{B}_0\mathcal{Y}_0 \nonumber \\
& =\mathlarger \sum_{j=1}^{n} \mathlarger \sum_{k=1}^{m_j} \begin{bmatrix}
    0 & 0 & \cdots & 0 \\
    0 & 0 & \cdots & 0 \\
    \vdots & \vdots & \ddots & \vdots \\
    B_k^{(j)} & 0 & \cdots & 0 \\
\end{bmatrix} \begin{bmatrix}
    \frac{-1}{(\lambda_0-a_j)^k}v\\
    \frac{-1}{(\lambda_0-a_j)^{k-1}}v\\
    \vdots\\
    \frac{-1}{(\lambda_0-a_j)}v\\
\end{bmatrix} \nonumber \\
&+ \begin{bmatrix}
    0 & I & 0 & \cdots & 0\\
    0 & 0 & I & \cdots & 0\\
    \vdots & \vdots & \vdots & \ddots & \vdots \\    
    C_0 & C_1 & C_2 & \cdots & C_{m-1}  
\end{bmatrix}\begin{bmatrix}
    v\\
    \lambda_0v\\
    \vdots\\
    \lambda_0^{m-1}v
\end{bmatrix} \nonumber \\
& = \sum_{j=1}^{n} \sum_{k=1}^{m_j} \begin{bmatrix}
    0\\
    0\\
    \vdots \\
    \frac{-B_k^{(j)}v}{(\lambda_0-a_j)^k}\\
\end{bmatrix}+\begin{bmatrix}
    \lambda_0v\\
    \lambda_0^2 v\\
    \vdots \\
    C_0v+C_1\lambda_0v+ \cdots + C_{m-1}\lambda_0^{m-1}v
\end{bmatrix} \nonumber\\ 
& =\begin{bmatrix}
    \lambda_0v\\
    \lambda_0^2v\\
    \vdots\\
    \lambda_0^mv\\
\end{bmatrix}
=\lambda_0 \begin{bmatrix}
    v\\
    \lambda_0v\\
    \vdots\\
    \lambda_0^{m-1}v\\
\end{bmatrix}=\lambda_0\mathcal{Y}_0 \label{Eqn-b}.
\end{align}
Therefore, from \eqref{Eqn-a} and \eqref{Eqn-b}, we get $\mathcal{C}_R \mathcal{Y} 
= \lambda_0 \mathcal{Y}$. Thus, if $\lambda_0$ is an eigenvalue of $R(\lambda)$, it 
is an eigenvalue of $\mathcal{C}_R$ also.

\end{proof}

\noindent
The following corollary to Theorem \ref{Thm- eigenvalues- matrix rational function} 
is immediate.

\begin{corollary}\label{Cor-zeros of rational function}
Let $r(\lambda)$ be a scalar complex rational function as in Equation 
\eqref{Eqn-scalar rational function-2} and let $\mathcal{C}_r$ be the associated 
block matrix as described in Equation \eqref{Eqn-block matrix for scalar rational fcn}. 
Then the zeros of $r(\lambda)$ are eigenvalues of $\mathcal{C}_r$.
\end{corollary}

\subsection{Bounds on the zeros of scalar complex rational functions}\label{sec-3.2}

\

In this section we consider an arbitrary scalar complex rational function 
$r(\lambda)$ with complex coefficients and derive  bounds on its zeros using the 
associated block matrix $\mathcal{C}_r$ given in Equation \eqref{Eqn-block matrix 
for scalar rational fcn}. These bounds are later used in Theorem \ref{Thm summary} 
to obtain bounds on the moduli of eigenvalues of rational matrices.

\medskip
\noindent
In the theorem that follows we give bounds on the zeros of $r(\lambda)$ using the 
norm of the associated block matrix.

\begin{theorem}\label{Thm-bounds for zeros using norms}
Let $r(\lambda)$ be a scalar complex rational function as in Equation 
\eqref{Eqn-scalar rational function-2}. If $\lambda_0$ is a zero of $r(\lambda)$ then,
\begin{enumerate}
\item $\displaystyle  |\lambda_0| \leq \max_{1 \leq j \leq n} \bigg\{1+|a_j|, 
\sum_{j=1}^{n} \sum_{k=1}^{m_j}|b_k^{(j)}|+ \sum_{i=0}^{m-1}|c_i| \bigg\}$.

\item {\small
 \[
|\lambda_0| \leq 
\begin{cases}
\max_{\substack{1 \leq j \leq n \\ 1 \leq i \leq m-1}} \bigg\{|a_j| + |b^{(j)}|, 
1 + |c_i|, |c_0| + n \bigg\}, \text{if each pole is of order $1$}, \\
\max_{\substack{1 \leq j \leq n \\ 1 \leq k \leq m_j \\ 1 \leq i \leq m-1}} 
\bigg\{|a_j| + |b_k^{(j)}|, 1 + |a_j|, 1 + |c_i|, |c_0| + \sum_{j=1}^{n}m_j \bigg\},  
\text{otherwise}.
\end{cases}
\]}

\end{enumerate}
\end{theorem}

\begin{proof}
Since $\lambda_0$ is a zero of $r(\lambda)$, by Corollary \ref{Cor-zeros of rational function} 
$\lambda_0$ is an eigenvalue of the corresponding block matrix $\mathcal{C}_r$. One 
can easily see that, 
\begin{align*}
||\mathcal{C}_r||_\infty \displaystyle= \max_{1 \leq j \leq n} \bigg\{1+|a_j|, 
\sum_{j=1}^{n} \sum_{k=1}^{m_j}|b_k^{(j)}|+ \sum_{i=0}^{m-1}|c_i| \bigg\}.
\end{align*}
We also have 
$||\mathcal{C}_r||_1 \displaystyle = \max_{\substack{1 \leq j \leq n \\ 1 \leq i \leq m-1}} 
\bigg\{|a_j|+|b^{(j)}|, 1+|c_i|, |c_0|+ n 
\bigg\}$, if each $a_j$ is a pole of order one and equals 
$\displaystyle \max_{\substack{1 \leq j \leq n \\ 1 \leq k \leq m_j\\ 1 \leq i 
\leq m-1}} \bigg\{|a_j|+|b_k^{(j)}|, 1+|a_j|, 1+|c_i|, |c_0|+ \sum_{j=1}^{n}m_j \bigg\}$ 
otherwise. Hence the proof.
\end{proof}

\medskip
\noindent
As a particular case the result for scalar 
complex polynomials is given in the following corollary.
\begin{corollary}\label{Cor-bounds for zeros using norms}
Let $p(\lambda) = \lambda^m -c_{m-1}\lambda^{m-1} - \dots -c_1 \lambda -c_0$. If 
$\lambda_0$ is a zero of $p(\lambda)$ then,
\begin{enumerate}
\item $|\lambda_0| \leq \max \Big\{1, \displaystyle \sum_{i=0}^{m-1}|c_i| \Big\}$.
\item $|\lambda_0| \leq \max\limits_{1 \leq i \leq m-1} \{|c_0|, 1+|c_i| \}$.
\end{enumerate} 
\end{corollary}
\begin{proof}
The block matrix corresponding to $p(\lambda)$ is $\mathcal{C}_p = \begin{bmatrix}
0 & 1 & 0 & \cdots & 0\\
0 & 0 & 1 & \cdots & 0\\
\vdots & \vdots & \vdots & \ddots & \vdots\\
0 & 0 & 0 & \cdots & 1\\
c_0 & c_1 & c_2 & \cdots & c_{m-1}\\
\end{bmatrix}$. It is easy to check that $||\mathcal{C}_p||_\infty = \max 
\Big\{1, \displaystyle \sum_{i=0}^{m-1}|c_i| \Big\}$ and $||\mathcal{C}_p||_1 = 
\max\limits_{1 \leq i \leq m-1} \{|c_0|, 1+|c_i| \}$. The desired result follows 
from this. 
\end{proof}

In the two results that follow, we use bounds on the numerical radius of the
associated block matrix to give bounds on the zeros of $r(\lambda)$.

\begin{theorem}\label{Thm-bounds on zeros using numerical range-2}
Let $r(\lambda)$ be a scalar complex rational function as in Equation 
\eqref{Eqn-scalar rational function-2}, where the $a_j$'s are nonnegative reals 
for $1 \leq j \leq n$. If $\lambda_0$ is a zero of $r(\lambda)$, then 
\begin{align*}
|\lambda_0| \leq \displaystyle \frac{\alpha+ \beta +\sqrt{(\alpha-\beta)^2 + 
(\gamma+\delta)^2}}{2},
\end{align*}
where $\alpha = \displaystyle \max_{1 \leq j \leq n} \Big\{a_j + 
\cos\Big(\frac{\pi}{m_j+1}\Big) \Big\}, \ \beta = w(\mathcal{B}_0), \ \gamma = 
\displaystyle \sqrt{\sum_{j=1}^{n} m_j}$ and \\ 
$\delta = \displaystyle 
\sqrt{ \sum_{j=1}^{n}\sum_{k=1}^{m_j} {|b_{k}^{(j)}|}^2}$ with $\mathcal{B}_0 = 
\begin{bmatrix}
0 & 1 & 0 & \cdots & 0\\
0 & 0 & 1 & \cdots & 0\\
\vdots & \vdots & \vdots & \ddots & \vdots\\
0 & 0 & 0 & \cdots & 1\\
c_0 & c_1 & c_2 & \cdots & c_{m-1}\\ 
\end{bmatrix}_{m \times m}$.
\end{theorem}

\begin{proof}
$\lambda_0$ being a zero of $r(\lambda)$, it follows from Corollary 
\ref{Cor-zeros of rational function} that $\lambda_0$ is an eigenvalue of the 
corresponding block matrix $\mathcal{C}_r= \begin{bmatrix}
\begin{array}{c c c c|c}
\mathcal{A}_n & 0 & \cdots & 0 & -\mathcal{F}_n \\
0 & \mathcal{A}_{n-1} & \cdots & 0 & -\mathcal{F}_{n-1} \\
\vdots & \vdots & \ddots & \vdots & \vdots \\
0 & 0 & \cdots & \mathcal{A}_1 & -\mathcal{F}_1 \\
\hline
\mathcal{B}_n & \mathcal{B}_{n-1} & \cdots & \mathcal{B}_1 & \mathcal{B}_0
\end{array}
\end{bmatrix} = \begin{bmatrix}
A & C \\
B & D \\
\end{bmatrix}$. By \eqref{Eqn-numerical radius inq} we have, 
\begin{align*}
w(\mathcal{C}_r) & \leq \frac{w(A)+ w(D) +\sqrt{(w(A)-w(D))^2 + (||C||_2+||B||_2)^2}}{2}.
\end{align*} 
We also have from Lemma \ref{Lem-numerical radius-Jordan block} that 
$w(A)= \displaystyle \max_{1 \leq j \leq n} 
\Big\{a_j + \cos\Big(\frac{\pi}{m_j+1}\Big) \Big\}$. 
Note that for $1 \leq j \leq n$ and $1 \leq k \leq m_j$, $B_k^{(j)}{B_k^{(j)}}^* 
= \begin{bmatrix}
0 & 0 & \cdots & 0 \\
0 & 0 & \cdots & 0 \\
\vdots & \vdots & \ddots & \vdots \\
0 & 0 & \cdots & {|b_{k}^{(j)}|}^2
\end{bmatrix}_{m \times m}$. 
Therefore $\mathcal{B}_j\mathcal{B}_j^* = \begin{bmatrix}
B_{m_j}^{(j)} & B_{m_j-1}^{(j)} & \cdots & B_1^{(j)}
\end{bmatrix}\begin{bmatrix}
{B_{m_j}^{(j)}}^* \\ 
{B_{m_j-1}^{(j)}}^* \\
\vdots \\ 
{B_1^{(j)}}^*
\end{bmatrix} = B_{m_j}^{(j)}{B_{m_j}^{(j)}}^*+ B_{m_j-1}^{(j)}{B_{m_j-1}^{(j)}}^* 
+ \cdots + B_1^{(j)}{B_1^{(j)}}^* = \begin{bmatrix}
0 & 0 & \cdots & 0 \\
0 & 0 & \cdots & 0 \\
\vdots & \vdots & \ddots & \vdots \\
0 & 0 & \cdots & \displaystyle \sum_{k=1}^{m_j} {|b_{k}^{(j)}|}^2
\end{bmatrix}_{m \times m}$. Then $B B^* = \begin{bmatrix}
\mathcal{B}_n & \mathcal{B}_{n-1} & \cdots & \mathcal{B}_1
\end{bmatrix} \begin{bmatrix}
\mathcal{B}_n^* \\ 
\mathcal{B}_{n-1}^* \\
\vdots \\
\mathcal{B}_1^*
\end{bmatrix}= \mathcal{B}_n\mathcal{B}_n^*+ \cdots + \mathcal{B}_1\mathcal{B}_1^* 
=\begin{bmatrix}
0 & 0 & \cdots & 0 \\
0 & 0 & \cdots & 0 \\
\vdots & \vdots & \ddots & \vdots \\
0 & 0 & \cdots & \displaystyle \sum_{j=1}^{n}\sum_{k=1}^{m_j} {|b_{k}^{(j)}|}^2
\end{bmatrix}$.
Thus $||B||_2 = \sqrt{\displaystyle \sum_{j=1}^{n}\sum_{k=1}^{m_j} {|b_{k}^{(j)}|}^2}$. 
Similarly, $||C||_2= \displaystyle \sqrt{\sum_{j=1}^{n}\sum_{k=1}^{m_j}1} = 
\sqrt{\sum_{j=1}^{n} m_j}$. This completes the proof.
\end{proof}

\noindent
The following corollary is a particular case for scalar complex 
polynomials. Note that this bound is the same as the bound given in Theorem $7$ of 
\cite{Abdurakhmanov}.

\begin{corollary}\label{Cor-bounds on zeros using numerical range-2}
Let $p(\lambda) = \lambda^m -c_{m-1}\lambda^{m-1} - \dots -c_1 \lambda -c_0$. If 
$\lambda_0$ is a zero of $p(\lambda)$ then 
\begin{equation}
|\lambda_0| \leq \frac{1}{2} \left(\cos \tfrac{\pi}{m}+|c_{m-1}|+ 
\sqrt{(\cos \tfrac{\pi}{m}-|c_{m-1}|)^2 +\left(1+ \sqrt{\sum_{i=0}^{m-2}|c_i|^2} 
\right)^2} \right)
\end{equation} 
\end{corollary}
\begin{proof}
The corresponding block matrix of $p(\lambda)$ can be partitioned as 
\begin{equation}
\mathcal{C}_p = \begin{bmatrix}
\begin{array}{c c c c|c}
0 & 1 & 0 & \cdots & 0\\
0 & 0 & 1 & \cdots & 0\\
\vdots & \vdots & \vdots & \ddots & \vdots\\
0 & 0 & 0 & \cdots & 1\\
\hline
c_0 & c_1 & c_2 & \cdots & c_{m-1}\\
\end{array}
\end{bmatrix} = \begin{bmatrix}
A & C \\
B & D
\end{bmatrix}.
\end{equation} 
The result then follows from the inequality \eqref{Eqn-numerical radius inq}.  
\end{proof}

We now compare the bounds obtained in Corollaries \ref{Cor-bounds for zeros using norms} 
and \ref{Cor-bounds on zeros using numerical range-2} with the bounds available in 
the literature. There are several bounds for the zeros of scalar complex polynomials 
of the form $p(\lambda)=\lambda^m -c_{m-1}\lambda^{m-1} - \dots - c_1\lambda -c_0$, 
where $c_i \in \mathbb{C}$, $1 \leq i \leq m-1$. We list a few of them here. Let 
$\lambda_0$ be any zero of $p(\lambda)$.

\begin{enumerate}
\item The following bound is due to Cauchy \cite{Horn-Johnson}:
\begin{equation}\label{eq-Cauchy}
|\lambda_0| \leq 1 + \max \{|c_i| : 0 \leq i \leq m-1 \}.
\end{equation}
\item The following bound is due to Carmichael and Mason \cite{Horn-Johnson}:
\begin{equation}\label{eq-Carmichael and Mason}
|\lambda_0| \leq \displaystyle \sqrt{1+ \sum_{i=0}^{m-1}|c_i|^2}.
\end{equation}
\item The bound due to Montel \cite{Horn-Johnson} says:
\begin{equation}\label{eq-Montel}
|\lambda_0| \leq \max \Bigg\{1, \sum_{i=0}^{m-1} |c_i| \Bigg\}.
\end{equation}
\item More recently Frakis, Kittaneh and Soltani \cite{Abdelkader-Faud-Soumia} obtained 
the following estimate:
\begin{smaller}
\begin{equation}\label{eq-Frakis, Kittaneh and Soltani}
|\lambda_0| \leq  \frac{1}{4} \left(2+ \sqrt{2}w(|A|+i|A^*|)+ \sqrt{\Big(\sqrt{2}w(|A|+i|A^*|) 
-2 \Big)^2 +4 \left(1+\sqrt{\sum_{i=0}^{m-3}|c_i|}\right)^2} \right),
\end{equation}
\end{smaller}
where $A = \begin{bmatrix}
c_{m-1} & c_{m-2} \\
1 & 0
\end{bmatrix}$, $|A| = (A^*A)^{\frac{1}{2}}$ and $w(A)$ denotes numerical radius of 
matrix $A$.
\item Another estimate due Cauchy \cite{Cauchy} which is a direct consequence of 
Rouch\'e's theorem is the following:
\begin{equation}\label{eq-Cauchy-Rouche}
|\lambda_0| \leq \eta,
\end{equation}
where $\eta$ is the unique positive zero of the real polynomial $u(x) = x^m - 
|c_{m-1}|x^{m-1} - \cdots - |c_1| x -|c_0|$.
\item Using Holder's inequality, Aziz and Rather \cite{Aziz-Rather} proved the following:
\begin{equation}\label{eq-Aziz-Rather}
|\lambda_0| \leq (m+1)^{\frac{1}{q}} \Bigg \{ \displaystyle \sum_{j=0}^{m} \bigg| 
\frac{tc_j-c_{j-1}}{t^{m-j}}\bigg|^p \Bigg\} ^{\frac{1}{p}}
\end{equation}
with $p>1, q>1, \frac{1}{p}+\frac{1}{q}=1, t>0$ and $a_{-1} =0$.
\end{enumerate}

\noindent
Note that the bound given in $(1)$ of Corollary \ref{Cor-bounds for zeros using norms} 
is the same as the bound due to Montel given in Inequality \eqref{eq-Montel}. We now 
show that depending on the modulus of the constant $c_0$, the bound obtained in $(2)$ 
of Corollary \ref{Cor-bounds for zeros using norms} is less than or equal to the bound 
due to Cauchy given in Inequality \eqref{eq-Cauchy}.

\noindent
Case $1$: If $\max\limits_{0 \leq i \leq m-1} \{|c_i| \} = |c_0|$, 
then 
\begin{center}
$\max\limits_{1 \leq i \leq m-1} \{|c_0|, 1+|c_i| \} \leq 1+|c_0| = 1+
\max\limits_{0 \leq i \leq m-1} \{|c_i| \}$.
\end{center}
Case $2$: If $\max\limits_{0 \leq i \leq m-1} \{|c_i| \} = |c_k|$ for $k \neq 0$, 
then 
\begin{center}
$\max\limits_{1 \leq i \leq m-1} \{|c_0|, 1+|c_i| \} = 1+|c_k| = 1+
\max\limits_{0 \leq i \leq m-1} \{|c_i| \}$.
\end{center}
In general one cannot determine which method will give the best bound among the bounds 
discussed above for scalar complex polynomials. However, specific polynomials are 
considered and bounds are compared in Example \ref{ex-polynomial-1} of Section \ref{sec-4}.

\subsection{Bounds on the moduli of eigenvalues $R(\lambda)$ using a scalar real rational function} 

\

In this section, given a rational matrix $R(\lambda)$, we associate a scalar 
real rational function $q(x)$ whose coefficients are composed of the norms of the 
coefficient matrices of $R(\lambda)$. We prove that a particular zero of $q(x)$ 
is a bound on the moduli of eigenvalues of $R(\lambda)$. We also make use of 
bounds on the zeros of scalar rational functions obtained in Section \ref{sec-3.2} 
to give bounds on the moduli of eigenvalues of $R(\lambda)$. We begin with the 
following lemma.
 
\begin{lemma}\label{Lem-zeros of rational function}
Let $q(x) = x^m-\alpha_{m-1}x^{m-1}-\cdots - \alpha_1 x - \alpha_0 - \displaystyle 
\sum_{j=1}^{n} \sum_{k=1}^{m_j} \frac{\beta_k^{(j)}}{(x-\gamma_j)^k}$ be a scalar real 
rational function, where the $\beta_k^{(j)}$'s are positive real numbers, $\alpha_i$'s 
and $\gamma_j$'s are nonnegative real numbers such that $\gamma_1 < \gamma_2 < \cdots 
< \gamma_n$. Then $q(x)$ has a zero $R$ such that $ \gamma_n < R$.
\end{lemma}

\begin{proof}
Note that $q(x)$ is continuous in the open interval $(\gamma_n, + \infty)$. Consider
\begin{align*}
\displaystyle \lim_{x \rightarrow \gamma_n^+}q(x) & = \displaystyle 
\lim_{x \rightarrow \gamma_n^+}\Bigg\{ x^m-\alpha_{m-1}x^{m-1}-\cdots - 
\alpha_1 x - \alpha_0 - \displaystyle 
\sum_{j=1}^{n} \sum_{k=1}^{m_j} \frac{\beta_k^{(j)}}{(x-\gamma_j)^k} \Bigg\}\\
& = \gamma_n^m- \sum_{i=1}^{m-1} \alpha_{i}\gamma_n^{i}-\displaystyle 
\sum_{j=1}^{n-1} \sum_{k=1}^{m_j} \frac{\beta_k^{(j)}}{(\gamma_n-\gamma_j)^k} - 
\displaystyle \lim_{x \rightarrow \gamma_n^+} \sum_{k=1}^{m_n} 
\frac{\beta_k^{(n)}}{(x-\gamma_n)^k} \\
& = - \infty
\end{align*} 
and 
\begin{align*}
\displaystyle 
\lim_{x \rightarrow + \infty} q(x) & = \displaystyle \lim_{x \rightarrow +
\infty}\Bigg\{ x^m-\alpha_{m-1}x^{m-1}-\cdots - \alpha_1 x - \alpha_0 - \displaystyle 
\sum_{j=1}^{n} \sum_{k=1}^{m_j} \frac{\beta_k^{(j)}}{(x-\gamma_j)^k} \Bigg\}\\
& = \displaystyle \lim_{x \rightarrow +\infty} \Bigg \{x^m \bigg( 1-\frac{a_{m-1}}{x} 
- \cdots -\frac{a_{1}}{x^{m-1}} -\frac{a_{0}}{x^m}\bigg) - \displaystyle 
\sum_{j=1}^{n} \sum_{k=1}^{m_j} \frac{\beta_k^{(j)}}{(x-\gamma_j)^k} \Bigg\} \\
& = +\infty.
\end{align*}
Therefore, by the intermediate value theorem $q(x)$ has a zero in $(\gamma_n, + \infty)$. 
This completes the proof.
\end{proof}

\medskip

\begin{theorem}\label{Thm-eigenvalue bounds-matrix rational function}
Let $R(\lambda)$ be a $p \times p$ rational matrix as in Equation 
\eqref{Eqn- rational matrix}. Associate a scalar real rational function to 
$R(\lambda)$ as follows:
\begin{equation}\label{Eqn-associated rational function}
q(x) = x^m-||C_{m-1}||x^{m-1}-\cdots - ||C_1|| x - ||C_0|| - \displaystyle 
\sum_{j=1}^{n} \sum_{k=1}^{m_j} \frac{||B_k^{(j)}||}{(x- |a_j|)^k},
\end{equation}
where $||\cdot ||$ is any induced matrix norm. If $\lambda_0$ is an eigenvalue of 
$R(\lambda)$ then $|\lambda_0| \leq R$, where $R$ is a real zero of $q(x)$ such that 
$|a_j| < R$ for all $j=1,\cdots,n$.
\end{theorem}

\begin{proof}
Without loss of generality assume that $|a_1| < |a_2| < \cdots < |a_n|$. By Lemma 
\ref{Lem-zeros of rational function}, $q(x)$ has a real zero $R$ such that $|a_j| < R$ 
for all $j=1,\cdots,n$. We 
then have 
\begin{align*}
0 & = q(R)=R^m - ||C_{m-1}|| R^{m-1} - \cdots - ||C_1|| R -||C_0||-\displaystyle 
\sum_{j=1}^{n} \sum_{k=1}^{m_j} \frac{||B_k^{(j)}||}{(R- |a_j|)^k}.
\end{align*} 
This implies $R^m  = ||C_{m-1}|| R^{m-1} + \cdots + ||C_1|| R + ||C_0||+\displaystyle 
\sum_{j=1}^{n} \sum_{k=1}^{m_j} \frac{||B_k^{(j)}||}{(R- |a_j|)^k}$. Therefore
\begin{align*}
1 = R^{-1}||C_{m-1}||  + \cdots + R^{-(m-1)} ||C_1|| + R^{-m} ||C_0||
+\displaystyle R^{-m} \sum_{j=1}^{n} \sum_{k=1}^{m_j} \frac{||B_k^{(j)}||}{(R- |a_j|)^k}.
\end{align*}
Suppose $\lambda_0 \in \mathbb{C}$ is such that $|\lambda_0| > R$. Then 
$|\lambda_0-a_j|\geq |\lambda_0|-|a_j|>R-|a_j|$. Since $|a_j| < R$, we have 
$\displaystyle \frac{1}{(|\lambda_0-a_j|)^k} < \frac{1}{(R-|a_j|)^k}$ for all 
$j=1,2,\cdots,n$ and $k=1,2,\cdots,m_j$. For any unit vector $v \in \mathbb{C}^n$ 
consider, 
\begin{align*}
& \Big|\Big| - \lambda_0^{-1}C_{m-1}v - \cdots - \lambda_0^{-(m-1)} C_1 v - 
\lambda_0^{-m} C_0 v+ \displaystyle \lambda_0^{-m} \sum_{j=1}^{n} \sum_{k=1}^{m_j} 
\frac{B_k^{(j)}}{(\lambda_0-a_j)^k} v \Big|\Big|\\
& \leq |\lambda_0|^{-1} ||C_{m-1}|| + \cdots + 
|\lambda_0|^{-m} ||C_0|| + \displaystyle |\lambda_0|^{-m} \sum_{j=1}^{n} \sum_{k=1}^{m_j} 
\frac{||B_k^{(j)}||}{(|\lambda_0-a_j|)^k} \\
& < R^{-1} ||C_{m-1}|| + \cdots + R^{-(m-1)} ||C_1|| + R^{-m} ||C_0|| + \displaystyle 
R^{-m} \sum_{j=1}^{n} \sum_{k=1}^{m_j} \frac{||B_k^{(j)}||}{(R-|a_j|)^k} \\
& =1.
\end{align*}
Now consider 
\begin{align*}
& ||R(\lambda_0)v|| \\
& =\Big|\Big| \lambda_0^m v- C_{m-1}\lambda_0^{m-1}v - \cdots -C_1 \lambda_0 v - 
C_0 v+ \displaystyle \sum_{j=1}^{n} \sum_{k=1}^{m_j} 
\frac{B_k^{(j)}}{(\lambda_0-a_j)^k} v\Big|\Big| \\
& = |\lambda_0|^m\Big|\Big|v- \lambda_0^{-1}C_{m-1}v - \cdots - 
\lambda_0^{-m} C_0 v+ \displaystyle \lambda_0^{-m} 
\sum_{j=1}^{n} \sum_{k=1}^{m_j} \frac{B_k^{(j)}}{(\lambda_0-a_j)^k} v \Big|\Big| \\
& \geq |\lambda_0|^m\Big(1-\Big|\Big| - \lambda_0^{-1}C_{m-1}v - \cdots - 
\lambda_0^{-m} C_0 v+ \displaystyle \lambda_0^{-m} 
\sum_{j=1}^{n} \sum_{k=1}^{m_j} \frac{B_k^{(j)}}{(\lambda_0-a_j)^k} v \Big|\Big| 
\Big)\\
& > |\lambda_0|^m(1-1) = 0.
\end{align*}
Therefore if $\lambda_0\in \mathbb{C}$ is an eigenvalue of $R(\lambda)$, then 
$|\lambda_0| \leq R$. This completes the proof.
\end{proof}

\medskip
In the following theorem we use Theorems \ref{Thm-bounds for zeros using norms} and 
\ref{Thm-bounds on zeros using numerical range-2} to obtain bounds on the moduli of 
eigenvalues of a rational matrix.

\begin{theorem}\label{Thm summary}
Let $R(\lambda)$ be a $p \times p$ rational matrix as in Equation 
\eqref{Eqn- rational matrix}. If $\lambda_0$ be an eigenvalue of $R(\lambda)$, then,

\begin{enumerate}
\item $\displaystyle  |\lambda_0| \leq \max_{1 \leq j \leq n} \bigg\{1+|a_j|, 
\sum_{j=1}^{n} \sum_{k=1}^{m_j}||B_k^{(j)}||+ \sum_{i=0}^{m-1}||C_i|| \bigg\}$.

\item {\small \[
|\lambda_0| \leq 
\begin{cases}
  \displaystyle\max_{\substack{1 \leq j \leq n \\ 1 \leq i \leq m-1}} \bigg\{|a_j| 
  + \big\|B^{(j)}\big\|, 1 + \big\|C_i\big\|, \big\|C_0\big\| + n \bigg\}, 
  \text{if each pole is of order $1$}, \\
  \displaystyle\max_{\substack{1 \leq j \leq n \\ 1 \leq k \leq m_j \\ 1 \leq i 
  \leq m-1}} \bigg\{|a_j| + \big\|B_k^{(j)}\big\|, 1 + |a_j|, 1 + \big\|C_i\big\|, 
  \big\|C_0\big\| + \sum_{j=1}^{n}m_j \bigg\}, \text{otherwise}.
\end{cases}
\]}

\

\item $|\lambda_0| \leq \displaystyle \frac{\alpha+ \beta +\sqrt{(\alpha-\beta)^2 + 
(\gamma+\delta)^2}}{2}$,

where $\alpha = \displaystyle \max_{1 \leq j \leq n} \Big\{|a_j| + 
\cos\Big(\frac{\pi}{m_j+1}\Big) \Big\}$, $\beta = w(\mathcal{B}_0)$, $\gamma =
\sqrt{\displaystyle \sum_{j=1}^{n} m_j}$ and $\delta = \sqrt{ \displaystyle 
\sum_{j=1}^{n}\sum_{k=1}^{m_j} {||B_{k}^{(j)}||}^2}$ with $\mathcal{B}_0 = 
\begin{bmatrix}
0 & 1  & \cdots & 0\\
0 & 0 & \cdots & 0\\
\vdots  & \vdots & \ddots & \vdots\\
0 & 0 & \cdots & 1\\
||C_0|| & ||C_1|| & \cdots & ||C_{m-1}||\\ 
\end{bmatrix}$.
\end{enumerate}
\end{theorem}

\begin{proof}
$\lambda_0$ being an eigenvalue of $R(\lambda)$, it follows from Theorem 
\ref{Thm-eigenvalue bounds-matrix rational function} that $|\lambda_0| \leq R$, 
where $R$ is real zero of the associated scalar real rational function $q(x)$ such 
that $|a_j| < R$ for $1 \leq j \leq n$. Now $(1)$ and $(2)$ follow from Theorem 
\ref{Thm-bounds for zeros using norms}, and $(3)$ follow from Theorem 
\ref{Thm-bounds on zeros using numerical range-2}.
\end{proof}

\subsection{A special case} 

\

We now consider a scalar complex rational function for which $p(\lambda)$ is a 
linear polynomial. That is, consider a scalar complex rational function of the 
following form
\begin{equation}\label{Eqn- rational function}
r(\lambda) = \lambda - c_0 - \displaystyle \sum_{j=1}^{n} \sum_{k=1}^{m_j} 
\frac{b_k^{(j)}}{(\lambda-a_j)^k}.
\end{equation}

We then have the following bound on the zeros of $r(\lambda)$.

\begin{theorem}\label{Thm-bounds for zeros}
Let $r(\lambda)$ be a scalar complex rational function as given in Equation 
\eqref{Eqn- rational function}. If $\lambda_0$ is a zero of $r(\lambda)$, then 
$|\lambda_0| \leq \displaystyle \max_{1\leq j \leq n} \Big\{|a_j| + 
\cos\bigg(\frac{\pi}{m_j+1}\bigg), |c_0| \Big\} + \frac{1}{2} \Bigg( 
\sqrt{\sum _{j=1}^{n}m_j} + \sqrt{\sum_{j=1}^{n} \sum_{k=1}^{m_j}|b_k^{(j)}|^{2}} \Bigg)$.
\end{theorem}

\begin{proof}
We know from Corollary \ref{Cor-zeros of rational function} that $\lambda_0$ is 
an eigenvalue of $\mathcal{C}_r$. Therefore 
$|\lambda_0| \leq w(\mathcal{C}_r)$. Split $\mathcal{C}_r$ as $\mathcal{C}_r 
= \mathcal{C}_1 + \mathcal{C}_2$, where $\mathcal{C}_1 = 
\text{diag}(\mathcal{A}_n,\mathcal{A}_{n-1}, \cdots , \mathcal{A}_1, \mathcal{B}_0)$  
and $\mathcal{C}_2 = \begin{bmatrix}
0 & 0 & \cdots & 0 & -\mathcal{F}_n \\
0 & 0 & \cdots & 0 & -\mathcal{F}_{n-1} \\
\vdots & \vdots & \ddots & \vdots & \vdots \\
0 & 0 & \cdots & 0 & -\mathcal{F}_1 \\
\mathcal{B}_n & \mathcal{B}_{n-1} & \cdots & \mathcal{B}_1 & 0
\end{bmatrix}$. Then $w(\mathcal{C}_r) \leq w(\mathcal{C}_1) + 
w(\mathcal{C}_2)$. It is easy to see that $w(\mathcal{C}_1) \leq \displaystyle 
\max_{1 \leq j \leq n} \Big\{|a_j| + \cos\Big(\frac{\pi}{m_j+1}\Big), |c_0| \Big\}$. 
Let $\mathcal{C}_2 = \begin{bmatrix}
0 & 0 & \cdots & 0 & -\mathcal{F}_n \\
0 & 0 & \cdots & 0 & -\mathcal{F}_{n-1} \\
\vdots & \vdots & \ddots & \vdots & \vdots \\
0 & 0 & \cdots & 0 & -\mathcal{F}_1 \\
\mathcal{B}_n & \mathcal{B}_{n-1} & \cdots & \mathcal{B}_1 & 0
\end{bmatrix} = \begin{bmatrix}
0 & \mathcal{F} \\
\mathcal{B} & 0
\end{bmatrix}$, where $\mathcal{F} = -\begin{bmatrix}
\mathcal{F}_n \\
\mathcal{F}_{n-1} \\
\vdots \\
\mathcal{F}_1
\end{bmatrix}$ and $ \mathcal{B} = \begin{bmatrix}
\mathcal{B}_n & \mathcal{B}_{n-1} & \cdots & \mathcal{B}_1
\end{bmatrix}$. By Lemma \ref{Lem-numerical radius-2} we have, $w(\mathcal{C}_2) 
\leq \displaystyle \frac{1}{2} (||\mathcal{F}||_2+ ||\mathcal{B}||_2) = \frac{1}{2} 
\Bigg( \sqrt{\sum _{j=1}^{n}m_j} + \sqrt{\sum_{j=1}^{n} 
\sum_{k=1}^{m_j}|b_k^{(j)}|^2} \Bigg)$. Hence $|\lambda_0| \leq \displaystyle 
\max_{1 \leq j \leq n} \Big\{|a_j| + \cos\bigg(\frac{\pi}{m_j+1}\bigg), |c_0| \Big\}\\
 + \frac{1}{2} \Bigg( \sqrt{\sum _{j=1}^{n}m_j} + \sqrt{\sum_{j=1}^{n} 
\sum_{k=1}^{m_j}|b_k^{(j)}|^2} \Bigg)$. 
\end{proof}

\medskip
As a corollary we have the following result for rational matrices whose matrix 
polynomial part is linear.

\begin{corollary}\label{Cor-bound for eigenvalue-linear}
Let $R(\lambda) = I \lambda  - C_0 + \displaystyle \sum_{j=1}^{n} \sum_{k=1}^{m_j} 
\frac{B_k^{(j)}}{(\lambda-a_j)^k}$ be a $p\times p$ rational matrix. If 
$\lambda_0$ is an eigenvalue of $R(\lambda)$, then $|\lambda_0| \leq \displaystyle  
\max_{1\leq j \leq n} \Big\{|a_j| + \cos\bigg(\frac{\pi}{m_j+1}\bigg), ||C_0|| \Big\} + 
\frac{1}{2} \Bigg( \sqrt{\sum _{j=1}^{n}m_j} + \sqrt{\sum_{j=1}^{n} 
\sum_{k=1}^{m_j}||B_k^{(j)}||^{2}} \Bigg)$.
\end{corollary}

\medskip
\section{Numerical results}\label{sec-4}\hspace*{\fill}
 
In this section, we determine bounds on the moduli of eigenvalues of an 
REP and bounds on the zeros of a scalar rational function using the above proposed 
methods. We also determine bounds on zeros of certain scalar polynomials. Among 
several bounds given above for the moduli of eigenvalues of an arbitrary rational 
matrix, the bound obtained using Theorem \ref{Thm-eigenvalue bounds-matrix rational function} 
is best. This is because the bound given in Theorem 
\ref{Thm-eigenvalue bounds-matrix rational function} is a zero of $q(x)$, 
while the other estimations are given using bounds on the zeros of $q(x)$. In 
general, we cannot compare other bounds, unless a particular example is considered. 
Similarly we cannot determine which method gives a better bound for zeros of scalar 
rational functions and scalar polynomials. All the calculations are done using MATLAB. 
In Example \ref{example-4.1} we have used the spectral norm to calculate the bounds given.   

\begin{example}\label{example-4.1}
Consider the following $3 \times 3$ rational matrix, 
\begin{align*}
R(\lambda) = I \lambda^3+C_1 \lambda +C_0 +\displaystyle 
\frac{B_1}{(\lambda-1)^2} + \frac{B_2}{\lambda-2} 
\end{align*}
where $C_0 =\begin{bmatrix}
4 & 1 & 1 \\
1 & 4 & 1 \\
1 & 1 & 4
\end{bmatrix}$, $C_1 =\begin{bmatrix}
2 & 1 & 1 \\
1 & 2 & 1 \\
1 & 1 & 2
\end{bmatrix}$ and $B_1 = B_2 = I$. The polynomial part $P(\lambda)= 
I \lambda^3+C_1 \lambda +C_0$ is taken from \cite{Bani-Fuad-Rawan}. 
The maximum of the moduli of eigenvalues of $R(\lambda)$ is $2.29$. 
The bounds obtained using Theorems \ref{Thm-eigenvalue bounds-matrix rational function} 
and \ref{Thm summary} are given in Table \ref{table:1} below.

\begin{table}[h!]
\centering
\begin{tabular}{l l l l l l}
\hline
Result & \hspace{5cm} & Bound \\ [0.5ex] 
\hline

Theorem \ref{Thm-eigenvalue bounds-matrix rational function} & \hspace{5cm} & 2.64\\
Theorem \ref{Thm summary} (1) & \hspace{5cm} & 12.00\\
Theorem \ref{Thm summary} (2) & \hspace{5cm} & 9.00\\
Theorem \ref{Thm summary} (3) & \hspace{5cm} & 4.99\\[1ex] 
\hline 
\end{tabular}
\vspace{0.2cm}
\caption{Bounds obtained by Theorems \ref{Thm-eigenvalue bounds-matrix rational function} 
and \ref{Thm summary}}
\label{table:1}
\end{table} 
\end{example}

To the best of our knowledge, there are currently no findings in the 
literature concerning the determination of bounds on the moduli of eigenvalues 
of rational matrices. However, one can convert the given rational matrix into a 
matrix polynomial and use well known results in literature on matrix polynomials 
to obtain bounds on the moduli of eigenvalues. We use this approach to the 
rational matrix given in Example \ref{example-4.1} and compare our bounds with a 
bound due to Theorem $2.1$ of \cite{Bini} and bounds given in \cite{Le-Du-Nguyen}. 
Instead of stating all of the results, we provide the theorem or corollary number. 
On multiplying $(\lambda -1)^2 (\lambda-2)$ to $R(\lambda)$, 
we obtain the following matrix polynomial 
\begin{equation}
P(\lambda) = (\lambda -1)^2 (\lambda-2) R(\lambda)
= A_6\lambda^6 +A_5 \lambda^5 +A_4\lambda^4 + A_3\lambda^3 + A_2\lambda^2 +A_1\lambda+A_0,
\end{equation}
where $A_0 =-(2C_0+I)$, $A_1 =(5C_0-2C_1-I)$, $A_2 =(5C_1-4C_0+I)$, $A_3 =
(C_0-2I-4C_1)$, $A_4 =(5I+C_1)$, $A_5 =-4I$, and $A_6 =I$. 
Using item $(4)$ of Theorem $2.1$ in \cite{Bini}, the bound on the moduli of 
eigenvalues of $P(\lambda)$ is the unique positive real root of the polynomial 
$p(x) =x^6-||A_5|| x^5 -||A_4||x^4 - ||A_3||x^3 - ||A_2||x^2 -||A_1||x-||A_0||$. 
But this polynomial is the same as the one given in Theorem $3.1$ of \cite{Le-Du-Nguyen}. 
Hence the bound will be the same as the bound obtained from Theorem $3.1$ of 
\cite{Le-Du-Nguyen}.

\begin{table}[h!]
\centering
\begin{tabular}{l l l l l l}
\hline
Results from \cite{Le-Du-Nguyen} & \hspace{5cm} & Bounds\\ [0.5ex] 
\hline
Theorem $3.1$  & \hspace{5cm} & 5.91 \\
Theorem $3.2$  & \hspace{5cm} & 22.00 \\
Theorem $3.3$  & \hspace{5cm} & 22.00 \\
Theorem $3.4$  & \hspace{5cm} & 7.32 \\
Corollary $3.4.2$  & \hspace{5cm} & 9.16 \\
Corollary $3.4.4$  & \hspace{5cm} & 8.94 \\
Corollary $3.4.6$  & \hspace{5cm} & 10.59 \\[1ex]
\hline
\end{tabular}
\vspace{0.2cm}
\caption{Bounds obtained by using results from \cite{Le-Du-Nguyen}.}
\label{table:2}
\end{table} 
\noindent
Tables \ref{table:1} and \ref{table:2} suggest that for Example \ref{example-4.1} 
the bounds obtained using Theorem \ref{Thm-eigenvalue bounds-matrix rational function} 
and Theorem \ref{Thm summary} $(3)$ from our manuscript are best.

\begin{example}\label{example-4.2}
The following is a scalar rational function whose polynomial part 
is taken from \cite{Abdelkader-Faud-Soumia}
\begin{align*}
r(\lambda) = \lambda^5 + \lambda +4- \displaystyle \frac{1}{\lambda -1}+ 
\frac{2}{(\lambda -1)^2}+\frac{3}{\lambda -3}+\frac{4}{(\lambda -3)^2}-
\frac{1}{(\lambda -3)^2}.
\end{align*}
Bounds on the zeros of $r(\lambda)$ obtained using Theorems 
\ref{Thm-bounds for zeros using norms} and 
\ref{Thm-bounds on zeros using numerical range-2} are given in Table \ref{table:3} 
below. The maximum of the moduli of zeros of $r(\lambda)$ is $3.12$.

\begin{table}[h!]
\centering
\begin{tabular}{l l l l l l}
\hline
Results & \hspace{5cm} & Bounds \\ [0.5ex] 
\hline

$(1)$ of Theorem \ref{Thm-bounds for zeros using norms} & \hspace{5cm} & 16.00\\
$(2)$ of Theorem \ref{Thm-bounds for zeros using norms}  & \hspace{5cm} & 9.00\\
Theorem \ref{Thm-bounds on zeros using numerical range-2} & \hspace{5cm} & 6.96\\
$(1)$ of Corollary \ref{Cor-bounds for zeros using norms} & \hspace{5cm} & 1179.00\\
$(2)$ of Corollary \ref{Cor-bounds for zeros using norms} & \hspace{5cm} & 364.00\\
Corollary \ref{Cor-bounds on zeros using numerical range-2} & \hspace{5cm} & 266.74\\
Cauchy \cite{Horn-Johnson} & \hspace{5cm} & 364.00\\
Carmichael and Mason \cite{Horn-Johnson} & \hspace{5cm} & 520.55\\
Montel \cite{Horn-Johnson} & \hspace{5cm} & 1179.00\\
Frakis et al. \cite{Abdelkader-Faud-Soumia} & \hspace{5cm} & 277.47\\
Rouch\'e \cite{Cauchy} & \hspace{5cm} & 14.60\\
Aziz and Rather \cite{Aziz-Rather} & \hspace{5cm} & 1956.37\\[1ex] 
\hline
\end{tabular}
\vspace{0.2cm}
\caption{Bounds on the moduli of zeros of $r(\lambda)$}
\label{table:3}
\end{table} 

Note that the zeros of $r(\lambda)$ can be found by expressing it as
$\displaystyle \frac{p(\lambda)}{q(\lambda)}$, 
where $p(\lambda)$ and $q(\lambda)$ are scalar polynomials such that $\gcd(p(\lambda), 
q(\lambda))=1$. Then the zeros of $r(\lambda)$ and $p(\lambda)$ are the same. One 
can now derive bounds on zeros of $p(\lambda)$ using various techniques discussed 
in Section $3.2$ for polynomials. In practice, however, it is difficult to determine 
the coefficients of $p(\lambda)$ and the coefficients might be very large. Table 
\ref{table:3} suggests that for the above example Theorem 
\ref{Thm-bounds on zeros using numerical range-2} and $(2)$ of Theorem 
\ref{Thm-bounds for zeros using norms} give best bounds.

\end{example}

\noindent
In the following example we compare our bounds with the existing bounds 
on zeros of scalar polynomials.

\begin{example}\label{ex-polynomial-1}
Let $p_1(\lambda) = \lambda^3- 2i \lambda^2 - (1+i)\lambda-1$ and 
$p_2(\lambda) = \lambda^3 - \lambda^2 - \lambda+2$. The maximum of the moduli of 
zeros of $p_1(\lambda)$ is $1.44$ and that of $p_2(\lambda)$ is $1.29$. The bounds 
given in Corollaries \ref{Cor-bounds for zeros using norms}, 
\ref{Cor-bounds on zeros using numerical range-2} and the bounds which are mentioned 
in the paragraph that follows Corollary \ref{Cor-bounds on zeros using numerical range-2}
are presented in Table \ref{table:4}. To calculate the bound given in Inequality \eqref{eq-Aziz-Rather}, 
we have chosen $p=2=q$ and $t=1$. For $p_1$ bound due to Rouch\'e's theorem is better. 
For $p_2$ the bounds obtained using $(2)$ of Corollary \ref{Cor-bounds for zeros using norms} 
and bound due to Rouch\'e's theorem are better.  

\begin{table}[h!]
\centering
\begin{tabular}{l l l l l l}
\hline
Results & Bounds on zeros of $p_1$ & Bounds on zeros of $p_2$ \\ [0.5ex] 
\hline

$(1)$ of Corollary \ref{Cor-bounds for zeros using norms} & 4.41 & 4.00\\
$(2)$ of Corollary \ref{Cor-bounds for zeros using norms} & 3.00 & 2.00\\
Corollary \ref{Cor-bounds on zeros using numerical range-2} & 2.81 & 2.48\\
Cauchy \cite{Horn-Johnson} & 3.00 & 3.00\\
Carmichael and Mason \cite{Horn-Johnson} & 2.83 & 2.65\\
Montel \cite{Horn-Johnson} & 4.41 & 4.00\\
Frakis et al. \cite{Abdelkader-Faud-Soumia} & 2.83 & 2.84\\
Rouch\'e \cite{Cauchy} & 2.67 & 2.00\\
Aziz and Rather \cite{Aziz-Rather} & 6.00 & 8.25\\[1ex] 
\hline
\end{tabular}
\vspace{0.2cm}
\caption{Bounds on the moduli of zeros of polynomials $p_1$ and $p_2$.}
\label{table:4}
\end{table}   
\end{example} 

\newpage
\medskip
\noindent
{\bf Declarations:} The authors declare that there is no conflict of interest in 
this work. 

\noindent
{\bf Data Availability Statement:} The authors declare that there is no data used 
in this research.

\noindent
{\bf Acknowledgements:} Pallavi Basavaraju and Shrinath Hadimani are grateful for
the financial assistance provided by the Council of Scientific and Industrial 
Research (CSIR) and the University Grants Commission (UGC), both departments of the 
Indian government.

\bibliographystyle{amsplain}

\end{document}